\documentclass[12pt]{amsart}
\usepackage{kmath,kerkis} 
\usepackage[T1]{fontenc}
\usepackage{amscd,amssymb, amsmath, wasysym}
\usepackage{graphicx}
\usepackage{url}
\usepackage[usenames, dvipsnames]{color}
\usepackage{hyperref}
 \hypersetup{
  colorlinks=true,
   citecolor=Plum,
   linkcolor=OliveGreen,
   urlcolor=blue}
   \usepackage{chngcntr} 
\counterwithout{subsection}{section}

\usepackage{amsrefs}
\usepackage[margin=1.00in]{geometry}
\theoremstyle{plain}
\newtheorem{thm}{Theorem}

\newtheorem*{thm*}{Theorem A'}

\newcommand\sO{{\mathcal O}}

\title[Derived Category of Enriques Surfaces]
{A Note on the Derived Category of Enriques Surfaces in Characteristic 2}

\author{Sofia Tirabassi}

\address{Department of Mathematics\\ University of Bergen, All\'egaten 41, Bergen, Norway}
\email{\url{sofia.tirabassi@uib.no}}
\keywords{Derived Category, Enriques Surfaces, Positive Characteristic}
\subjclass[2010]{Primary 14F05 Secondary 14J20, 14F05}

\begin{document}
 \maketitle
\begin{abstract}
 We show that the (twisted) derived category ``recognizes'' the three different kinds of Enriques surfaces in characteristic 2.
\end{abstract}
\section{Introduction}
This note concerns the study of derived categories of Enriques surfaces defined over an algebraically closed field $k$ of characteristic 2.\par
Derived categories were introduced by Verdier in the late 60s, with the aim to extend known theorems in algebraic geometry to a relative setting. Thanks to the work of Bondal, Orlov, Kawamata, Bridgeland, Katzarkov, Bayer--Macr\`i, Bayer--Macr\`i--Toda, and many others, it is nowadays understood that they provide a versatile and powerful tool for better understanding the geometry of algebraic varieties. However, the great majority of the techniques employed in the aforementioned studies are tailored for complex varieties (or more generally for varieties over an algebraically closed field of characteristic 0), and the research on the more algebraic setting of varieties defined over (possibly finite) fields of positive characteristic is still, in many aspects, at an embryonic stage.\par
In this work we concentrate on Enriques surfaces over fields of characteristic 2. By the Bombieri--Mumford classification of algebraic surfaces (\cite{BM3}), an \emph{Enriques surface} $X$ is a smooth projective surface with numerically trivial canonical bundle, second Betti number equal to 10 and Euler characteristic 1. Over fields of characterstic two they are split into three kinds: \emph{classical}, which have a torsion canonical bundle of order 2, \emph{non-classical ordinary}, whose canonical bundle is trivial and the whose Frobenius morphism acts isomorphically on $H^1(X,\mathcal{O}_X)$, and \emph{non-classical supersingular}, which have again a trivial canonical bundle and are such that their Frobenius induces the zero-morphism on $H^1(X,\sO_X)$. In both these latter two cases we have that $H^1(X,\sO_X)$ is a 1-dimensional vector space. The central result of this note is the following:

\begin{thm}\label{Main}
 Let $X$ be an Enriques surface defined over an algebraically closed field $k$ of characteristic 2. If $Y$ is is a smooth $k$-scheme with the same derived category of $X$, then $Y$ is an Enriques surfaces of the same kind as $X$.
\end{thm}
The main characters of our proof are the invariance under derived equivalence of the order of the canonical bundle and the Roquier isomorphism (\cite{RO2011}). Using the same ingredients, we also prove a twisted variant  of Theorem \ref{Main}.
\subsection*{Notation and Terminology}
Given a smooth projective variety $Z$, we will denote by $\mathbf{D}^b(Z)$ its bounded derived category of coherent sheaves. The symbol $\mathcal{T}_Z$ will stand for the tangent bundle of $Z$. We say that two varieties $Z_1$ and $Z_2$ are \emph{derived equivalent} or \emph{Fourier--Mukai partners} if there exists an exact equivalence $\Phi:\mathbf{D}^b(Z_1)\rightarrow\mathbf{D}^b(Z_2)$.
\subsection*{Acknowledgements}I would like to thank Prof. C. Liedtke for asking me the question which this note answers at a conference in Berlin. I am also grateful to the unknown referee for many suggestions and improvements. It was him (or her) who encouraged me to pursue the variant of Theorem \ref{Main} proposed in the last section. Finally, I am indebted to K. Honigs for very interesting mathematical conversation and for pointing me the difficulties of using the HKR isomorphism in characteristic 2.
\section{The Proof}
We first prove that $Y$ is an Enriques surface by using a similar argument to the one in \cite{HLT}*{Theorem 3.3}. The key ingredients are the invariance under derived equivalence of the order of the canonical bundle (see \cite{Hu}*{Proposition 4.1}) and the isomorphisms that Fourier--Mukai transforms induce at cohomological level which, imply that derived equivalent surface have the same Betti numbers.\\
By standard results on derived invariants, we know that $Y$ is a surface with numerically trivial canonical bundle. In particular it is a minimal surface of Kodaira dimension 0. Denote by $\Phi:\mathbf{D}^b(X)\rightarrow \mathbf{D}^b(Y)$ an exact equivalence. Let $l$ be an odd prime. By \cite{Ho}*{Lemma 3.1} the equivalence $\Phi$ descends to cohomology and yields and isomorphism
 $$H^0_{\acute{e}t} (X,\mathbf{Q}_l) \; \oplus \; H^2_{\acute{e}t} (X,\mathbf{Q}_l)(1)\oplus \; H^4_{\acute{e}t} (X,\mathbf{Q}_l)(2) \; \simeq \; H^0_{\acute{e}t} (Y,\mathbf{Q}_l)
 \; \oplus\;  H^2_{\acute{e}t} (Y,\mathbf{Q}_l)(1)\oplus \; H^4_{\acute{e}t} (Y,\mathbf{Q}_l)(2),$$
 where $(i)$ denotes the Tate twist. After computing dimensions, the above isomorphism  leads to the equality of Betti numbers:
 $$b_2(X) \;= \;b_2(Y) \; = \; 10.$$
 So we conclude that $Y$ is an Enriques surfaces by applying the Bombieri--Mumford classification theorem (see for example \cite{Li}*{Paragraph 7}).\par
 Now we assume that $X$ is a classical Enriques surface. Then its canonical bundle has order 2. As remarked before, since derived equivalences commute with Serre functors, also the canonical bundle of $Y$ will have order 2, and therefore $Y$ is a classical Enriques surface.\par
 Suppose conversely that the canonical bundle of $X$ (and so the one of $Y$) is trivial. The table on page 6 of \cite{Li2} tells us that one can distinguish non-classical ordinary and supersingular  Enriques surfaces by looking at the cohomology of their tangent bundle. Therefore, in order to conclude the proof, we need just to show that $H^0(X,\mathcal{T}_X)\simeq H^0(Y,\mathcal{T}_Y)$. This is a consequence of the Roquier isomorphism. More precisely, in \cite {RO2011} the author proves that the exact equivalence $\Phi$ induces an isomorphism of algebriaic groups
 \begin{equation}\label{Roquier}
  F_\Phi:\mathrm{Pic}^0(X)\rtimes \mathrm{Aut}^0(X)\rightarrow \mathrm{Pic}^0(Y)\rtimes \mathrm{Aut}^0(Y),
 \end{equation}
which in turns induces an isomorphism of their tangent spaces at the origin. Now, it is well known that, for any smooth projective variety $Z$, the tangent space in $\mathcal{O}_Z$ of $\mathrm{Pic}^0(Z)$ is isomorphic to $H^1(Z,\mathcal{O}_Z)$. On the other side there is an isomorphism of the tangent space at the origin of $\mathrm{Aut}^0(Z)$ with $H^0(Z,\mathcal{T}_Z)$ (cfr. \cite{MO1967}). Thus, since as scheme, the semi-direct product is simply the product, \eqref{Roquier} yields an isomorphism
$$H^1(X,\mathcal{O}_X)\oplus H^0(X,\mathcal{T}_X)\simeq H^1(Y,\mathcal{O}_Y)\oplus H^0(Y,\mathcal{T}_Y).$$ 
The statement is proved by taking dimension on both sides, since both $X$ and $Y$ are non-classical Enriques surfaces, and so  both $H^1(X,\mathcal{O}_X)$ and $H^1(Y,\mathcal{O}_Y)$ are one dimensional.\qed
\section{Variant: Twisted Fourier--Mukai Partners of Enriques Surfaces}
In this section we prove the following variant to Theorem \ref{Main}, involving the twisted derived categories introduced by C\u{a}ld\u{a}raru in his thesis \cite{CalThesis}.
\begin{thm*}\label{variant}
 Let $X_1$ and $X_2$ be Enriques surfaces over an algebraically closed field of characteristic 2, and consider  two Brauer classes, $\alpha_1$ and $\alpha_2$, with $\alpha_i\in \mathrm{Br}(X_i)$. If there is an exact equivalence of the \emph{twisted} bounded  derived categories
 $$\Phi:\mathbf{D}^b(X_1,\alpha_1)\rightarrow \mathbf{D}^b(X_2,\alpha_2),$$
 then $X_1$ and $X_2$ are of the same kind.
\end{thm*}
\begin{proof}[of Theorem \ref{variant}]
 We first remark that non-classical Enriques surfaces have a trivial Brauer group (\cite{CD2012}*{Proposition 5.3.5}). So if $X_1$ and $X_2$ are both non-classical then the statement can be deduced directly by Theorem \ref{Main}. Thus we can suppose that $X_1$ is classical. Then its Brauer group is isomorphic to $\mathbf{Z}/2\mathbf{Z}$ (\cite{CD2012}). In order to prove the theorem we have just to show that there cannot be an exact equivalence 
 $$\Phi:\mathbf{D}^b(X_1,\alpha_1)\rightarrow \mathbf{D}^b(X_2),$$
 where $\alpha_1$ is the only non-trivial class of $\mathrm{Br}(X_1)$ and $X_2$ is a non-classical Enriques surface. But, if such an equivalence should exist, by \cite{Nauer1}*{Theorem 1.6.15} or \cite{Nauer2}*{Theorem 22} we would have that
 $$0\simeq H^0(X_1,\omega_{X_1})\simeq H^0(X_2,\omega_{X_2}),$$
 which yields an obvious contradiction.
\end{proof}

\end{document}